\documentclass[a4paper,12pt]{amsart} 
\title[Saturations of binomial ideals]{The norm of the saturation of a binomial ideal, and applications to Markov bases}

\usepackage{amsmath, amssymb, mathrsfs, amsthm, shorttoc, stmaryrd}
\usepackage{mathtools} 
\usepackage{hyperref}

\usepackage{enumitem}

\usepackage{cleveref}
\crefformat{equation}{(#2#1#3)}
\let\oref\ref
\AtBeginDocument{\renewcommand{\ref}[1]{\cref{#1}}}

\numberwithin{equation}{subsection}

\usepackage{pgf,tikz}
\usetikzlibrary{matrix, calc, arrows}

\usepackage{tikz-cd} 

\makeatletter
\newcommand*{\doublerightarrow}[2]{\mathrel{
  \settowidth{\@tempdima}{$\scriptstyle#1$}
  \settowidth{\@tempdimb}{$\scriptstyle#2$}
  \ifdim\@tempdimb>\@tempdima \@tempdima=\@tempdimb\fi
  \mathop{\vcenter{
    \offinterlineskip\ialign{\hbox to\dimexpr\@tempdima+1em{##}\cr
    \rightarrowfill\cr\noalign{\kern.5ex}
    \rightarrowfill\cr}}}\limits^{\!#1}_{\!#2}}}
\newcommand*{\triplerightarrow}[1]{\mathrel{
  \settowidth{\@tempdima}{$\scriptstyle#1$}
  \mathop{\vcenter{
    \offinterlineskip\ialign{\hbox to\dimexpr\@tempdima+1em{##}\cr
    \rightarrowfill\cr\noalign{\kern.5ex}
    \rightarrowfill\cr\noalign{\kern.5ex}
    \rightarrowfill\cr}}}\limits^{\!#1}}}
\makeatother

\newcommand{\on}[1]{\operatorname{#1}}
\newcommand{\bb}[1]{{\mathbb{#1}}}

\newcommand{\ca}[1]{{\mathcal{#1}}}
\newcommand{\bd}[1]{{\mathbf{#1}}}

\newcommand{\ul}[1]{{\underline{#1}}}

\newcommand{\abs}[1]{\lvert#1\rvert}

\newcommand{\sub}{\subseteq}

\theoremstyle{definition}
\newtheorem{definition}{Definition}[section]

\theoremstyle{plain}

\newtheorem{proposition}[definition]{Proposition}
\newtheorem{lemma}[definition]{Lemma}
\newtheorem{theorem}[definition]{Theorem}
\newtheorem{corollary}[definition]{Corollary}

\theoremstyle{remark}

\newtheorem{remark}[definition]{Remark}

\renewcommand{\phi}{\varphi}

\author{David Holmes}
\date{\today}

\newcounter{nootje}
\newcommand\todo[1]{[*\thenootje]\marginpar{\tiny\begin{minipage}
{20mm}\begin{flushleft}\thenootje : 
#1\end{flushleft}\end{minipage}}\addtocounter{nootje}{1}}

\newcommand{\expanded}[1]{#1}
\renewcommand{\expanded}[1]{}

\newcommand{\mat}[1]{\begin{bmatrix}#1\end{bmatrix}}
\newcommand{\laurent}[2]{[#1_1^{\pm 1}, \dots, #1_{#2}^{\pm 1}]}

\begin{document}
\maketitle
\begin{abstract} 
Given a pure binomial ideal $I$ in variables $x_i$, we define a new measure of the complexity of the saturation of $I$ with respect to the product of the $x_i$, which we call the \emph{norm}. We give a bound on the norm in terms of easily-computed invariants of the ideal. We discuss statistical applications both practical and theoretical. 
\end{abstract}


\tableofcontents

\section{Introduction}\label{sec:intro}

\subsection{Background}
Let $A$ be a $k \times r$ matrix with integer entries, and let $u \in \bb N^r$ be a vector with non-negative entries. The \emph{fibre containing $ u$} is defined as
\begin{equation}
\ca F( u) = \{ v \in \bb N^r : Au = Av\}. 
\end{equation}
Understanding the structure of this fibre is important in a number of statistical tests. For example, the vectors in $\bb N^r$ might represent tables of data, and the matrix $A$ might output the row and column sums of these tables, so the fibre consists of all tables with non-negative entries and with the same row and column sums as the starting table $u$. See \cite{Diaconis1998Algebraic-algor} for more details and examples. In particular, one often wants to generate samples from some probability distribution (often uniform or hypergeometric) on the fibre. If the fibre is small it is practical to simply enumerate all the elements of the fibre. However, in practical applications the fibre is often far too large to enumerate, and the standard approach is to perform a \emph{random walk} in the fibre, generating samples via the Metropolis-Hastings Markov-Chain Monte-Carlo algorithm. In order to perform a random work, we must upgrade the fibre into a graph (whose vertices are the elements of the fibre). The requirements for the Metropolis-Hastings algorithm are rather mild, the key condition is that the graph must be \emph{connected} (since the random walk will always remain within its starting connected component). 

\subsubsection{A random walk in the fibre}\label{sec:naive_algorithm}The most naive way to convert the fibre into a graph is to choose a generating set $B$ for the kernel $K\sub \bb Z^r$ of $A$, and then form a (simple, undirected) graph by putting an edge between distinct vertices $v_1$ and $v_2$ whenever $v_1 - v_2 \in B$ or $v_2 - v_1 \in B$. We say $\ca F(u)$ is \emph{connected by $B$} if the resulting graph is connected. In \ref{sec:examples} we will see several examples of $B$ that \emph{fail} to connect $\ca F(u)$. The major innovation of Diaconis and Sturmfels \cite{Diaconis1998Algebraic-algor} was to give an algorithm to construct a generating set $B$ which connects \emph{every} fibre of a given matrix $A$. 


\subsubsection{Saturated ideals and connected fibres}\label{sec:sat_ideals}
To describe their result, we need a little more notation. Given $b \in B$, we write $b = b^+ - b^-$, both summands having non-negative entries. In the ring $R = \bb Z[x_1, \dots, x_r]$ we form the elements 
\begin{equation}\label{eq:xg}
x^{b^+} \coloneqq \prod_{i=1}^r x_i^{b^+_i}, \;\;\; x^{b^-} \coloneqq \prod_{i=1}^r x_i^{b^-_i}, 
\end{equation}
and define an ideal $\ca I_B = (x^{b^+} - x^{b^-} : b \in B)\sub R$. Then the key theorem is:
\begin{theorem}[Diaconis-Sturmfels, \cite{Diaconis1998Algebraic-algor}]\label{thm:DS}
Fix a $k \times r$ matrix $A$, and let $B$ be a generating set for the integral kernel of $A$. Suppose the ideal $\ca I_B$ is saturated with respect to the element $x_1\cdots x_r \in R$. Then for every $u \in \bb N^r$, the fibre $\ca F(u)$ is connected by $B$. 
\end{theorem}
If $\ca I_B$ is saturated, $B$ is often called a \emph{Markov basis} (this should not be interpreted as implying linear independence of the elements of $B$). The theorem then tells us that we can generate samples according to our preferred distribution by following the naive random walk algorithm above using the basis $B$. 

On the other hand, suppose that we have a generating set $B$ such that $\ca I_B$ is not saturated. We can (at least in principal) apply a standard saturation algorithm to $\ca I_B$ to produce a saturated ideal, and moreover the generating set produced will in fact consist of pure difference binomials (i.e. differences of monomials; see \ref{def:binomial}). Reversing the procedure \ref{eq:xg} we can recover a new generating set $B'$ for the kernel $K$ of $A$, and following the above theorem of Diaconis-Sturmfels, this generating set will connect all fibres, enabling efficient sampling.  

Thus, when it is possible to compute this saturation, the problem is essentially solved. However, the standard algorithm for saturation involves the computation of $r$ Gr\"obner bases, and is at present only practical for relatively small examples (software to carry out such computations can be found at \url{4ti2.de}). 

\subsubsection{Connected fibres without saturation}\label{sec:AHT_algorithm}
The difficulty of computing the saturation motivated Aoki, Hara and Takemura \cite{Hara2012Running-Markov-} to suggest an algorithm for generating samples without needing to compute the saturation. They begin in the same say, with a generating set $B = \{b_1, \dots, b_n\}$ for the integral kernel, but instead of making moves consisting of addition or subtraction of a single element of $B$, they instead generate $n$ non-negative integers $a_i$ from a Poisson distribution with some chosen mean $\lambda$, and $n$ elements $\epsilon_i \in \{ +1, -1\}$, and their move consists of addition of $\sum_i \epsilon_i a_i b_i$ if the result lies in the fibre, and staying put otherwise. Since the Poisson distribution generates every non-negative integer with non-zero probability it is immediate that the resulting fibre is connected; in fact, the graph on the fibre is a complete graph, but with highly non-uniform probability of selecting moves from among edges. 

They then perform a number of numerical experiments with various values of $\lambda$. In cases where it was possible to compute the saturation, they show that for careful choice of $\lambda$ their algorithm performs comparably to that coming from a Markov basis, and they also illustrate that their algorithm can be applied in cases where the saturation is too hard to compute (though they can of course provide no guarantee that their algorithm is converging in reasonable time; it appears to do so, but this might be deceptive if the fibre has some connected components that are very hard to hit --- see \ref{sec:practical}). 

There is some tension in the use of this algorithm when it comes to choosing the value of $\lambda$. If one chooses $\lambda$ very large then the algorithm takes a long time before it (appears to) converge. On the other hand, a small value of $\lambda$ will product more rapid \emph{apparent} convergence, but there is a greater risk that one is simply failing to see one or more connected components of the fibre in the time for which the algorithm is run.

\subsection{Results}
\subsubsection{A bound on the complexity of the saturation}
In the light of the above discussion it is natural to try to bound how large and complex the saturation of the ideal $\ca I_B$ can get. To make this more precise, we define the \emph{norm} of the saturation: 

\begin{definition}\label{def:norm}
Let $B$ be set of $n \ge 1$ vectors in $\bb Z^r$. We write $\ca I_B$ for the ideal in $R = \bb Z[x_1, \dots, x_r]$ as defined in \ref{sec:sat_ideals}. The \emph{norm} of $B$ is the smallest integer $N \ge 1$ such that there exists a finite generating set $G$ for the saturation of $\ca I_B$ with respect to $x_1\cdots x_r$, with the properties that 
\begin{enumerate}
\item
Every element of $G$ is a pure difference binomial;
\item 
Every $g \in G$ can be written in the form 
\begin{equation}
g = \sum_{i=1}^N \epsilon_i m_i (x^{b_i^+} - x^{b_i^-}). 
\end{equation}
where the $\epsilon_i \in  \{-1,0,1\}$, the $m_i$ are monomials, and the $b_i$ are elements of $B$. 
\end{enumerate}
\end{definition}

The main result of this paper is the following explicit bound on the norm. In \ref{sec:improving_AHT,sec:alt_algo} we will show how this can be applied to give new algorithms for sampling from fibres without needing to compute the saturation. 


\begin{theorem}\label{thm:main}
Let $B$ be set of $n \ge 1$ vectors in $\bb Z^r$. Write $\beta$ for the maximum of the absolute values of the coefficients of elements of $B$. Then the norm of $B$ is at most 
\begin{equation}\label{eq:big_constant}
n(2n\beta)^{n-1}. 
\end{equation}
\end{theorem}
Our proof (see \ref{sec:proof_of_main})is constructive; it gives an algorithm to determine a generating set $F$ as in the definition of the norm. We do not know whether this algorithm could be practical; it is a-priori less efficient than procedures using Gr\"obner bases, but is highly parallelisable. 

The connection of the norm to fibre connectivity and Markov chains runs via the following result (proven in \ref{sec:proof_of_cor}):
\begin{proposition}\label{cor:connected}
Let $A$ be a $k \times r$ integer matrix, and $B = \{b_1, \dots, b_n\}$ a basis of the kernel, with $B$ having norm $N$. Let $u \in \bb N^r$, and construct a graph with vertex set the fibre $\ca F(u)$, and where we draw an edge from $v_1$ to $v_2$ if and only if $v_1 - v_2$ can be written as an integer linear combination
\begin{equation*}
v_1 - v_2 = \sum_{i=1}^n a_i b_i
\end{equation*}
with $\sum_{i=1}^n \abs{a_i} \le N$. Then this graph is connected. 
\end{proposition}

\begin{remark}
Given a $k \times r$ integral matrix $A$, note that it is easy to compute a basis $B$ of the integral kernel of $A$ from the Smith normal form of $A$. Indeed, if $SAT = D$ is the Smith normal form (so $S$ and $T$ are invertible, and $D$ diagonal with $D_{i,i}\mid D_{i+1, i+1}$), then let $1 \le j \le r$ be maximal such that $D_{j,j} \neq 0$. Then an integral basis of the kernel of $A$ is given by $Te_{j+1}, \dots, Te_r$, where $e_i$ is the $i$th standard basis vector in $\bb Z^r$. 

Conversely, while $B$ does not determine $A$, it does determine the fibres $\ca F(u)$, so the matrix $A$ is not really essential, but is very relevant to the statistical applications. 
\end{remark}

\subsubsection{Comparison to other results in the literature}

Needless to say, we are not the first to try to control the complexity of the saturation of an ideal in a polynomial ring. Indeed, the standard method of computing the saturation reduces to a Gr\"obner basis computation, whose efficient implementation has been the focus of too much research to begin to list here. Specialising to the case of binomial ideals, the literature is still much too large to give more than a quick glimpse of. There are general theoretical results on the structure of fibre graphs (\cite{Windisch2016Rapid-mixing-an}, \cite{Hemmecke2015On-the-connecti}, \cite{Windisch2019The-fiber-dimen}, \cite{Gross2013Combinatorial-d}, ...). There are also many results bounding the degree of the binomials appearing in the saturation, see (\cite[chapter 13]{Sturmfels1996Grobner-bases-a}, \cite{Haws2014Markov-degree-o}, \cite{Koyama2015Markov-degree-o}, ...), and bounding the \emph{Markov complexity}; this is defined in \cite{Santos2003Higher-Lawrence}, and studied in \cite{Charalambous2014Markov-complexi} and elsewhere. 

However, we are not aware on bounds on the `norm' (\ref{def:norm}) in the literature. Indeed, from an algebraic point of view it appears a rather unnatural invariant. The reason for studying it comes purely from the application (via \ref{cor:connected}) to fibre connectivity and Markov bases. In the remainder of \ref{sec:intro} we hope to justify it from this point of view, and perhaps motivate further research in this direction. An unusual feature of our results is that we do not utilise Gr\"obner bases of related techniques; this is not from dislike, but simply because we could not see how to bound the norm from that perspective; we hope that others may have more success. 

\subsubsection{Improving the algorithm of Aoki, Hara and Takemura}\label{sec:improving_AHT}

Aoki, Hara and Takemura connect the fibre by allowing arbitrarily large integer linear combinations of elements of the basis $B$. However, \ref{cor:connected} shows that it suffices to take combinations with coefficients bounded by the norm $N$ of $B$. This allows us to improve the efficiency of their algorithm, by truncating the Poisson distribution at $N$. A second algorithm they present (where the coefficients of the $b_i$ are chosen from a multinomial distribution) can be enhanced in a similar way. The bound on the norm coming from \ref{thm:main} is likely to be large in comparison with the chosen $\lambda$, so will not have a large impact on the runtime, but we hope that better bounds on the norm can be found in future. 

A more useful application might be to predicting good values of the constant $\lambda$ in their algorithm, or giving heuristic bounds on the convergence time for a given value of $\Lambda$. The norm $N$ can be seen as the maximum distance between connected components of the fibre, thus to be have a reasonable chance of hitting all components we should take a number of steps that is very large compared to $1/\bb P(\on{Poisson}_\lambda \ge N)$. 

%


\subsubsection{An alternative algorithm for connecting the fibres}\label{sec:alt_algo}

In the naive algorithm of \ref{sec:naive_algorithm}, one starts at a vector $v \in \ca F(u)$, and chooses at random an element $b \in \pm B$, and considers the step $v + b$. If $v + b$ in $\ca F(u)$ then this is returned as the next element of the Markov chain. If $v + b \notin \ca F(u)$, then the algorithm simply returns $v$. However, if we have a bound on the norm then we can modify the algorithm so that the fibre will always be connected; if $v + b \notin \ca F(u)$ then, rather than returning $v$, we choose another element $b_1$ from $\pm B$, and consider the vector $v + b + b_1$. If $v + b + b_1$ lies in $\ca F(u)$ we return is as the next step in the Markov chain, otherwise we repeat, until we either hit $\ca F(u)$ again, or we have taken $N$ consecutive steps outside the fibre, in which case we return $v$ again. Alternatively, this can be viewed as a weighted random walk in a certain graph with vertex set $\ca F(u)$. To define this graph, we first define a graph $\ca F_\bb Z(u)$ with vertex set $\{ v \in \bb Z^r : Au = Av\}$ and with an edge between $v_1$ and $v_2$ whenever $v_1 - v_2 \in \pm B$. Then we define a graph with vertex set $\ca F(u)$ by putting an edge between two vertices whenever they can be connected by a path in $\ca F_\bb Z(u)$ of length at most $N$, and which does not intersect $\ca F(u)$ except at its endpoints. Again, by \ref{cor:connected} this new graphs is guaranteed to be connected. 

More generally, with \ref{thm:main,cor:connected} in hand it is easy to propose new sampling algorithms which guarantee to connect the fibre. The challenge is to design algorithms with reasonable runtime, at least heuristically (rigorous runtime analysis seems hard but very interesting). 

If the fibre $\ca F(u)$ is large with respect to the norm $N$ then designing reasonably efficient algorithms is not hard, since the runtime will be dominated by time spent in the `interior' of the fibre. On the other hand, if the fibre is small compared to $N$ then the runtime will be dominated by time spent around the edge of the fibre looking for new connected components, and will depend sensitively on the norm (or more precisely, on our bound on the norm). 




\subsection{Practical consequences}\label{sec:practical}

The algorithm of \ref{sec:AHT_algorithm} is proven to converge. And in practise the Markov chain is often observed to settle down quite fast. Indeed, in practise it is the latter which will generally be relied upon; people run algorithms until the chain appears to converge. However, there is a critical problem here. Namely, we see in \ref{sec:eg_arb_bad_connected} examples where the chain will \emph{appear} to converge very rapidly, but this `apparent' limit will not be the true limit (the runtime required to achieve true convergence may easily be arranged to exceed the lifespan of the solar system). We hope that this kind of pathological behaviour will be very rare in practise, but at present this seems hard to verify. Our aim in this paper is to get an idea of how long the algorithm should be run in order to be reasonably confident that the `apparent limit' of the chain is in fact the true limit. We are not completely successful in this, partly because our bound on the norm is rather large for practical use (and probably not sharp), and also because passing from the bound in \ref{thm:main} to an estimate on the convergence time needs substantial further work. We think it is interesting and useful to investigate this further. In the meantime, we would encourage people this type of algorithm to let it run for as long as possible, even after the chain appears to have settled down, to maximise the change of hitting new connected components.

\subsection*{Acknowledgements}
This work owes its existence to a seminar on algebraic statistics organised in Leiden in the Autumn of 2018 by Garnet Akeyr, Rianne de Heide, and Rosa Winter. I am very grateful to them for organising it, for the many expert speakers who took the time to patiently explain basic ideas of probability and statistics to us, and especially to the determined participants who survived to the end, and offered very useful comments on a presentation of the results contained here. 

\section{Examples}\label{sec:examples}
\subsection{A very simple example}\label{sec:v_simple_example}

Consider the matrix 
\begin{equation*}
A = \mat{0&1&2&3\\
3&2&1&0}. 
\end{equation*}
An integral basis for the kernel of $A$ is then given by $B = \{b_1, b_1\}$ where 
\begin{equation*}
b_1 =  \mat{1\\-2\\1\\0}, \;\;\; b_2 = \mat{0 \\ 1\\-2\\1}. 
\end{equation*}
The fibre containing the vector $\mat{2&2&2&2}^T$ is illustrated in \ref{fig}, where red arrows correspond to addition of $b_1$, and blue arrows to addition of $b_2$. Evidently, this fibre is not connected, since the element $\mat{4&0&0&4}^T$ is isolated. Thus is our chain begins anywhere in the large component it will never hit the isolated vertex, and if it begins at the isolated vertex it will remain there. This has practical consequences, since it is common to simply run such a Markov chain until it appears (by eye) to have converged; in this example, convergence will be rapid, but the resulting distribution will not be the expected one (c.f. \ref{sec:practical}). 

The approach of Diaconis-Sturmfels is to replace the basis $B$ by a larger generating set which makes the fibre connected. The ideal $\ca I_B$ is generated by $x_1x_3 - x_2^2$ and $x_2x_4 - x_3^2$, and its saturation can be generated by these two polynomials together with the polynomial $x_1x_4 - x_2x_3$, the latter corresponding to the vector $\mat{1&-1&-1&1}^T$. Clearly one can step from $\mat{3&1&1&3}^T$ to $\mat{4&0&0&4}^T$ by addition of this new vector, so the fibre is indeed connected by this new generating set for the integral kernel of $A$. 

Our approach is to allow the chain to step briefly outside the fibre while it hunts for vectors with non-negative entries. As long as we allow two negative steps the fibre will become connected, as we can step from $\mat{3&1&1&3}^T$ to $\mat{4&0&0&4}^T$ via $\mat{4&-1&2 & 3}^T$ or $\mat{3 & 2 & -1 & 4}^T$; one sees easily that the norm is $2$. Let us compute the bound resulting from \ref{thm:main}: we have $\beta = 2$ and $n=2$, so our bound is $16$. Thus if we use the bound from the theorem we should allow 16 negative steps; it is clear that this will be sufficient to connect the fibre, but also that this bound is not optimal.

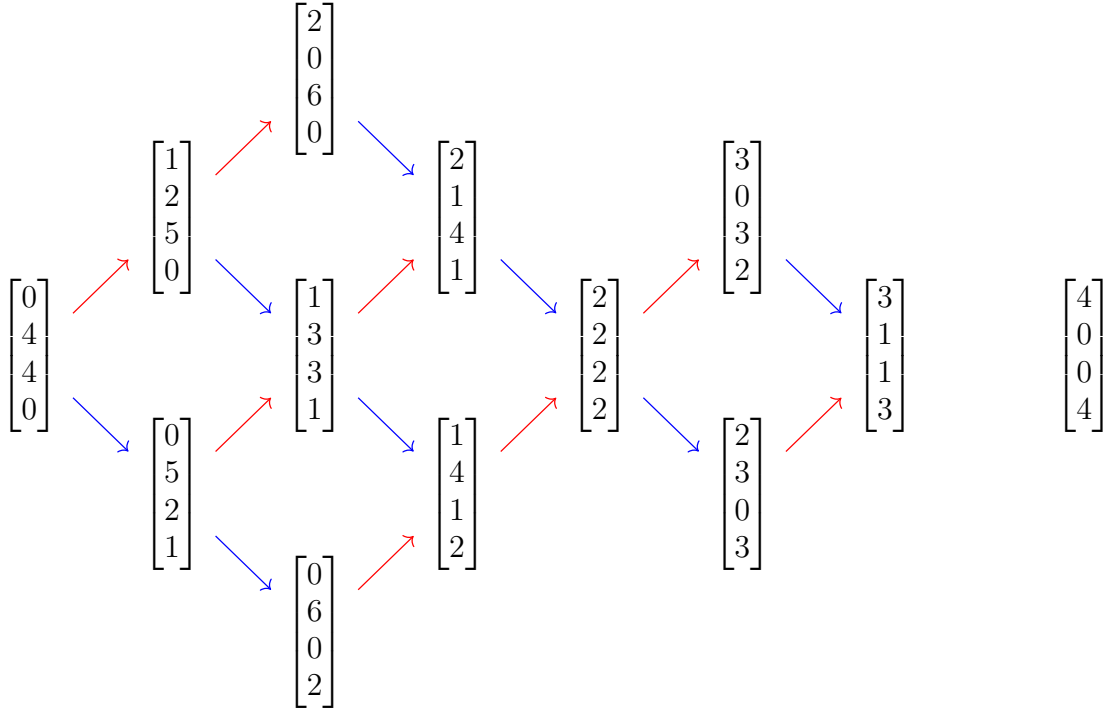
\begin{figure}\caption{A (non-connected) fibre }\label{fig}
 \begin{tikzcd}[row sep = -0.5cm, column sep = scriptsize]
 & & \mat{2\\0\\6\\0} \arrow[dr, blue]  & & & & & &\\
 & \mat{1\\2\\5\\0} \arrow[ur, red] \arrow[dr, blue] & & \mat{2\\1\\4\\1} \arrow[dr, blue]  & & \mat{3\\0\\3\\2}  \arrow[dr, blue]& & & \\
 \mat{0\\4\\4\\0} \arrow[ur, red] \arrow[dr, blue] & & \mat{1\\3\\3\\1} \arrow[ur, red] \arrow[dr, blue] & & \mat{2\\2\\2\\2} \arrow[ur, red] \arrow[dr, blue] && \mat{3\\1\\1\\3} & & \mat{4\\0\\0\\4} \\
&  \mat{0\\5\\2\\1} \arrow[ur, red] \arrow[dr, blue]  & & \mat{1\\4\\1\\2} \arrow[ur, red]  & & \mat{2\\3\\0\\3} \arrow[ur, red]  & & &\\
& & \mat{0\\6\\0\\2} \arrow[ur, red]  & & & & & &\\
\end{tikzcd}
\end{figure}

\subsection{Families where the fibres are arbitrarily badly connected}\label{sec:eg_arb_bad_connected}

Consider the $1 \times 3$ matrix $A = \mat{1 & 1 & 1}$, and write $e_i$ for the $i$th standard basis vector in $\bb Z^3$. Let $u = e_2$. Then the fibre $\ca F(u) = \{e_1, e_2, e_3\} $. For a positive integer $n$, choose the basis 
\begin{equation*}
B_n = \left\{ \mat{0\\1\\-1} , \mat{-1 \\ n \\ 1-n} \right\}
\end{equation*}
of the kernel of $A$. 
Then the fibre consists of two connected components, namely $\{ e_2, e_3\}$ and $\{e_1\}$. Moreover, to step between the connected components requires $(n-1)$ consecutive negative steps. Thus for every positive integer $M$ and every real number $\lambda$ there exists an integer $n$ such that the algorithm of Aoki, Hara and Takemura presented in \ref{sec:AHT_algorithm} applied to the above basis $B_n$ will appear to converge immediately, but will take $M$ steps before the probability of hitting the other connected component rises above any given positive threshold. 

This example is quite artificial, as the fibre is essentially simple, but we have made a poor choice of generating set $B_n$. We can also construct a slightly less artificial example of the same phenomenon, by generalising the example in \ref{sec:v_simple_example}. For an integer $n\ge 2$, let 
\begin{equation*}
A_n = \mat{ 1 & 2 & \cdots & n-1 & n\\ n & n-1 & \cdots & 2 & 1}, 
\end{equation*}
and consider the basis of the integral kernel given by 
\begin{equation*}
B_n = \left\{\mat{1\\-2\\1\\0\\0\\ \vdots \\ 0}, \mat{0 \\ 1\\-2\\1\\0\\ \vdots \\ 0}, \cdots, \mat{0 \\ 0 \\ \vdots\\ 1\\-2\\ 1}\right\}, 
\end{equation*}
where we denote the elements of $B_n$ by $b_2, \dots, b_{n-1}$ in the given order. Then the fibre of $\mat{2 & \cdots & 2}^T$ contains the vector $v = \mat{n &0 & \cdots & 0 & n}^T$. This vector $v$ is at least $n-2$ steps distant from any other point in the fibre; more precisely, if $c_1, \dots, c_r \in \pm B_n$ are such that 
\begin{equation*}
v + \sum_{i=1}^r c_i \in \ca F(v), 
\end{equation*}
then either $r \ge n-2$ or $v + \sum_{i=1}^r c_i = v$ (the bound $n-2$ is in fact sharp). We leave the elementary verification to the interested reader. Again we see that, though the algorithm of \ref{sec:AHT_algorithm} (and variants) may appear to converge rapidly, there are connected components which take an arbitrarily long time to hit.

\subsection{The no-three-factor-interaction model}
This model is described in detail (in particular, its statistical interpretation) in \cite{Aoki2012Markov-bases-in}. It depends on a choice of three positive integers $I$, $J$ and $K$; we will often take $I=J=K$ for simplicity. The matrix $A$ is then an ${(IJ + JK + KI) \times IJK}$ matrix, described in a slightly complicated way. Define $Id_I$ to be the $I \times I$ identity matrix, and $1_I$ to be a row vector of length $I$ with all entries equal to $1$. Then 
\begin{equation*}
 A = \mat{Id_I \otimes Id_J \otimes 1_K\\ Id_I\otimes 1_J \otimes Id_K \\ 1_I \otimes Id_J \otimes Id_K}, 
\end{equation*}
where $\otimes$ represents the Kroneker product of matrices. 

In \cite{Hara2012Running-Markov-}, the authors numerically test their algorithm (\ref{sec:AHT_algorithm}) on the no-three-factor-interaction model in the cases $I = J = K = 3,$ $ 5,$ and $10$. In the case $I=3$ the saturation can be computed by Gr\"obner basis techniques, but seems presently out of reach $I=5$, and worse for $I=10$. In each case they compute a basis for the integral kernel, then run numerical tests of their algorithm for several values of the Poisson parameter $\lambda$, and also occasionally replacing the Poisson with a different distribution (we are not completely clear on how they chose these parameters and distributions). In the case $I=3$ they compare their results to those obtained using a saturated basis, and observe that the Markov chains coming from their algorithm converge similarly to those coming from a saturated basis (though for $\lambda = 50$ the convergence is rather slow). 

For $I=10$ their algorithm does not converge well, but for $I=5$ it appears to converge fairly rapidly. As throughout this paper, the question we are interested in is whether this apparent convergence can be trusted, or is it possible that there is some connected component of the fibre which their chain has never hit? Of course, their algorithm will find every component with probability 1 if allowed to run for unlimited time, but there is no a-priori reason to assume that the time required for this will be in any way comparable to the time required for the chain to appear to settle down. 

To try to get a handle on this, let us compute our upper bound on the number of negative steps required to walk between components (the `distance between' connected components of the fibre). Using SAGE we compute the smith normal form of the $75 \times 125$ matrix $A$, obtaining an integral basis $B$ with $n=64$ elements. The largest absolute value of an entry in $B$ is $\beta = 1$. This leads to an upper bound on the norm by 
\begin{equation}
N' = n(2n\beta)^{n-1} = 64(128)^{63} \approx 3.6\times 10^{134}. 
\end{equation}

Now, in this example Aoki, Hara and Takemura replace the Poisson distribution with a geometric distribution (for reasons which are unclear to us), and try parameters $p = 0.1,$ $0.5$. The proportion of steps in their algorithm which will exceed $N'$ in length is then so small that it is likely never to occur before the sun runs cold. This means that \emph{if} the bound $N'$ were to be close to the true norm, then this algorithm will in practise never converge to the correct solution. In practise, our bound on the norm is surely very far from sharp, but we gave this example to illustrate the difficulty in guaranteeing convergence (despite the fact that the algorithm might appear to the human eye to have converged). 

\section{Proof of the main results}
\subsection{Proof of \ref{thm:main}}\label{sec:proof_of_main}
Let $B = \{b_1, \dots, b_n\}$ be a set of vectors in $\bb Z^r$. Following the notation of \ref{eq:xg}, we write
\begin{equation*}
f_i^+ = x^{b_i^+}, \;\;\; f_i^- = x^{b_i^-}, \;\;\; f_i = f_i^+ - f_i^- 
\end{equation*}
in the ring $R = \bb Z[x_1, \dots, x_r]$. Then $\ca I_B = (f_1, \dots, f_n) \sub R$, and our goal is to bound the norm of the saturation 
\begin{equation}
\on{Sat}_{x_1\cdots x_r} \ca I_B = \{ a \in R : \exists m > 0 :  a(x_1 \cdots x_r)^m \in  \ca I_B \}. 
\end{equation}
\begin{definition}\label{def:binomial}
A \emph{pure binomial} in $R$ is an element of the form $m_1 - m_2$ where the $m_i$ are monomials. An ideal $I \sub R$ is called \emph{pure binomial} if it admits a generating set consisting of pure binomials; evidently, $\ca I_B$ is a pure binomial ideal. 
\end{definition}
\begin{lemma}[\cite{Herzog2018Binomial-ideals}, proposition 3.18]\label{lem:sat_is_pure_binomial}
The saturation of $\ca I_B$ with respect to $x_1\cdots x_r$ is also a pure binomial ideal. 
\end{lemma}

\begin{definition}
Given pure binomials $f = f^+ - f^-$ and $g = g^+ - g^-$, we define the \emph{subtraction polynomial} (again a pure binomial)
\begin{equation*}
S(f, g) = g^+f + f^-g = f^+g^+ - f^-g^-. 
\end{equation*}
\end{definition}
If $f$, $g \in \ca I_B$ then clearly $S(f, g)$ lies in $\ca I_B$.  

We make the unsurprising notational conventions that $-- = +$, $+- = -+ = -$ and $++ = +$; thus we interpret $f^{--} = f^+$, which is less usual, but makes for efficient and hopefully comprehensible notation in what follows. 

\begin{definition}\label{def:iter_subtraction}
Let $\epsilon\colon \{1, \dots, n \} \to \{+, -\}$, and let $t\colon \{1, \dots, n\} \to \bb N$. Define
\begin{equation}
S(\epsilon, t) = \prod_{i=1}^n (f_i^{\epsilon(i)})^{t(i)} -  \prod_{i=1}^n (f_i^{-\epsilon(i)})^{t(i)} \in \ca I_B, 
\end{equation}
(here we use our `$- - = +$' convention when we write $f_i^{-\epsilon(i)}$). 
\end{definition}

\begin{lemma}\label{Q:structure_of_pure_binomials}
Let $P$ be a pure binomial in $\ca I_B$. Then there exist $\epsilon$, $t$, and monomials $m$ and $ n$ such that
\begin{equation*}
nP = mS(\epsilon, t). 
\end{equation*}
%
%
%
\end{lemma}
\begin{proof}
For the purposes of the proof, we will simplify notation by assuming that for every $b_i \in B$, the element $-b_i$ also lies in $B$. 

Let $P \in \ca I_B$ be a pure binomial. Write $P = \sum_{j=1}^k m_j f_{i_j}$, where the $m_j$ are monomials. We can and do assume that $k$ is chosen minimal, and we proceed by induction on $k$. The case $k=1$ is trivial. 

Up to harmless sign changes, there exists a $j_0$ such that $m_{j_0} f_{i_{j_0}}^+ = P^+$. Reordering, we may assume that $j_0=1$, so 
\begin{equation*}
P - m_1 f_{i_1} = \sum_{j=2}^k m_j f_{i_j}
\end{equation*}
is again a pure difference binomial. By the induction hypothesis there exist monomials $m$ and $n$ and vectors $\epsilon$, $t$ with 
\begin{equation*}
m \sum_{j=2}^k m_j f_{i_j} = n S(\epsilon, t). 
\end{equation*}
Write $S(\epsilon, t)  = S^+ - S^-$. Then 
\begin{equation*}
m P  =  nS^+ - nS^- + m_1 f_{i_1}^+ - m_1 f_{i_1}^-. 
\end{equation*}
This this is a binomial, up to signs we may assume without loss of generality that $nS^- = m_1f_{i_1}^+$. We can then write
\begin{equation*}
f_{i_1}^+ m P = n\left(f_{i_1}^+ S^+ - f_{i_1}^-S^-\right) = n S'
\end{equation*}
%
%
where $S'$ is an iterated subtraction binomial of the $f_i$. 
\end{proof}

\begin{theorem}\label{thm:binomials_from_subtraction}
There exist a positive integer $M$, functions $\epsilon_1, \dots, \epsilon_M$ and $t_1, \dots, t_M$ as in \ref{def:iter_subtraction}, and monomials $m_1, \dots, m_M \in R$, such that
\begin{enumerate}
\item
for all $1 \le j \le M$ we have $m_j \mid S(\epsilon_j, t_j)$;
\item 
\begin{equation*}
\on{Sat}_{x_1\cdots x_r} \ca I_B = \left( \frac{S(\epsilon_j, t_j)}{m_j} : 1 \le j \le M \right). 
\end{equation*}
\end{enumerate}
\end{theorem}
\begin{proof}
Combine \ref{lem:sat_is_pure_binomial} and \ref{Q:structure_of_pure_binomials}. 
\end{proof}

Given $t\colon \{ 1, \dots, n\} \to \bb N$ we define the $L^1$-length of $t$ to be the sum of its values. To prove \ref{thm:main} it suffices to show that we can choose each of the vectors $t_j$ in \ref{thm:binomials_from_subtraction} to have $L^1$-length  bounded by the constant $N$ of (\oref{eq:big_constant}). Given vectors $\epsilon$ of signs and $t$ of natural numbers as in \ref{def:iter_subtraction}, observe that the power of $x_j$ dividing $S(t, \epsilon)$ is given by
\begin{equation}\label{eq:power_of_x_dividing}
\min\left( \sum_{i=1}^n t(i) \on{ord}_{x_j} f_i^{\epsilon(i)}, \sum_{i=1}^n t(i) \on{ord}_{x_j} f_i^{-\epsilon(i)}  \right). 
\end{equation}
We say the \emph{minimum in \ref{eq:power_of_x_dividing} is achieved on the $+$ side} if 
\begin{equation*}
\sum_{i=1}^n t(i) \on{ord}_{x_j} f_i^{\epsilon(i)}\le \sum_{i=1}^n t(i) \on{ord}_{x_j} f_i^{-\epsilon(i)}, 
\end{equation*}
and we say the 
\emph{minimum in \ref{eq:power_of_x_dividing} is achieved on the $-$ side} if 
\begin{equation*}
\sum_{i=1}^n t(i) \on{ord}_{x_j} f_i^{\epsilon(i)}\ge \sum_{i=1}^n t(i) \on{ord}_{x_j} f_i^{-\epsilon(i)}. 
\end{equation*}

\begin{definition}
Given $\epsilon \colon \{1, \dots, n \} \to \{ +, -\}$ and $\delta\colon \{1, \dots, r\} \to \{ +, -\}$, we define
\begin{equation*}
T_{\epsilon, \delta} = \{t \in \bb N^n : \forall 1 \le i \le r, \text{ the minimum in \ref{eq:power_of_x_dividing} is achieved on the side } \delta(i) \}. 
\end{equation*}
\end{definition}
This set $T_{\epsilon, \delta}$ is a rational polyhedral cone in $\bb N^n$, and for fixed $\epsilon$ we have
\begin{equation}\label{eq:Tdelta_fill_up}
\bigcup_\delta T_{\epsilon, \delta} = \bb N^n. 
\end{equation}
Given $t \in T_{\epsilon, \delta}$, we write 
\begin{equation}
\phi_t = \frac{S(\epsilon, t)}{\prod_{j=1}^r x_j^{\sum_{i=1}^n t(i) \on{ord}_{x_j}f_i^{\epsilon(i)\delta(i)}}}, 
\end{equation}
which we write as a difference of monomials $\phi_t = \phi_t^+ - \phi_t^-$ in the usual way. From the definition of $T_{\epsilon, \delta}$ we see that $\phi_t \in R$, i.e. all exponents of the $x_i$ are non-negative. 

\begin{lemma}\label{lem:phi_t_in_ideal}
Fix $\epsilon$ and $\delta$ as above, and let $t, t_1, \dots, t_a \in T_{\epsilon, \delta}$ such that $t = t_1 + \cdots + t_a$. Then 
\begin{equation*}
\phi_t \in (\phi_{t_1}, \dots, \phi_{t_a}) \sub R. 
\end{equation*}
\end{lemma}
\begin{proof}
Elementary manipulations yield
\begin{equation}
\begin{split}
\phi_t & = \prod_{\alpha = 1}^a \phi_{t_\alpha}^+  - \prod_{\alpha = 1}^a \phi_{t_\alpha}^-\\
& = S( \cdots S(S(\phi_{t_1}, \phi_{t_2}) \phi_{t_3}) \cdots \phi_{t_a}). \qedhere
\end{split}
\end{equation}
\end{proof}

\begin{theorem}
For each $\epsilon$ and each $\delta$, choose a generating set $\tau_{\epsilon, \delta}$ for the cone $T_{\epsilon, \delta}$. Then
\begin{equation}\label{eq:generating_set}
\bigcup_{\epsilon, \delta} \{ \phi_t : t \in \tau_{\epsilon, \delta}\}
\end{equation}
is a generating set for $\on{Sat}_{x_1\cdots x_r} \ca I_B$. 
\end{theorem}
\begin{proof}
Let $t \in \bb N^n$, then $S(\epsilon, t) \in \ca I_B$, and $\phi_t \in R$, hence by definition of the saturation we see that $\phi_t \in \on{Sat}_{x_1\cdots x_r} \ca I_B$. Conversely, \ref{thm:binomials_from_subtraction} tells us that the $\phi_t$ generate $\on{Sat}_{x_1\cdots x_r} \ca I_B$ as $t$ ranges over $\bb N^n$. We must justify why it suffices to consider only $t$ ranging over the set in (\oref{eq:generating_set}). Fixing $\epsilon$, we note that every $t \in \bb N^r$ lies in some $T_{\epsilon, \delta}$ by (\oref{eq:Tdelta_fill_up}), and then by \ref{lem:phi_t_in_ideal} it suffices to range over elements of a generating set for $T_{\epsilon, \delta}$. 
\end{proof}

Fixing $\epsilon$ and $\delta$, it remains to show that $T_{\epsilon, \delta}$ can be generated by vectors of length bounded by the constant $N$ from (\oref{eq:big_constant}). First, we have the elementary
\begin{lemma}\label{lem:cone_gens}
Let $v_1, \dots, v_a \in \bb N^n$, and let $C$ be the intersection of $\bb N^n$ with the \emph{rational} cone spanned by the $v_i$. Then $C$ is generated by
\begin{equation*}
C \cap \left\{\sum_{i=1}^a \lambda_i v_i : \lambda_i \in [0, 1) \right\} \cup \{v_1, \dots, v_a\}. 
\end{equation*}
\end{lemma}
Observe that the faces of $T_{\epsilon, \delta}$ are defined by the equations
\begin{equation}\label{eq:hyperplanes}
\sum_{i=1}^n t(i) \on{ord}_{x_j} f_i ^{\epsilon(i)} = \sum_{i=1}^n t(i) \on{ord}_{x_j} f_i ^{-\epsilon(i)} , 
\end{equation}
thus the extremal rays of $T_{\epsilon, \delta}$ are obtained by solving $n-1$ equations of the form \ref{eq:hyperplanes}. Let $\beta$ be the maximum of the absolute values of the $\on{ord}_{x_j}f_i = b_{i,j}$ as $i$ and $j$ vary. By Siegel's lemma, the $L^1$-length of such a (non-zero) solution is then bounded above by
\begin{equation*}
(2n\beta)^{n-1}. 
\end{equation*}
From \ref{lem:cone_gens}, and cutting into simplicial cones, we see that $T_{\epsilon, \delta}$ can be generated by vectors of length at most $N = n(2n\beta)^{n-1}$, concluding the proof. 

\subsection{Proof of \ref{cor:connected}}\label{sec:proof_of_cor}

Let $G$ be a generating set for the saturation as in \ref{def:norm}. Each $g \in G$ is a pure difference binomial, say $g = x^{c^+} - x^{c^-}$ with $c^+$, $c^- \in \bb N^r$, and can be written in the form
\begin{equation*}
g = \sum_{i=1}^N \epsilon_i m_i f_{j_i}, 
\end{equation*}
with $\epsilon_i \in \{ 1, 0 , -1\}$, $m_i$ monomials, and $f_j$ as in \ref{sec:proof_of_main}. Writing $c = c^+ - c^-$, it suffices (by \ref{thm:DS}) to show that $c$ can be written as $c = \sum_{i=1}^n a_i b_i$ with $\sum_{i=1}^n \abs{a_i} \le N$. 

We wish to prove this by induction on $N$, but this makes no sense as $N$ is the norm. Instead we rephrase things slightly so that induction makes sense: 
\begin{lemma}
Let $M$ be a positive integer, and suppose that the expression
\begin{equation}\label{eq:sum_to_M}
\sum_{i=1}^M \epsilon_i m_i f_{j_i}, 
\end{equation}
is a pure binomial $x^{c^+} - x^{c^-}$, where $\epsilon_i \in \{ 1,  -1\}$, and the $m_i$ are monomials. Then there exist integers $a_1, \dots, a_n$ with $\sum_{i=1}^n \abs{a_i} \le M$ and $c^+ - c^- = \sum_{i=1}^n a_i b_i$. 
\end{lemma}
It is clear that the lemma (applied with $M = N$) implies \ref{cor:connected}, so it only remains to verify the lemma. 
\begin{proof}
For a warmup we treat first the case $M=1$. Then 
\begin{equation*}
x^{c^+} - x^{c^-} = \pm m(x^{b_{j_1}^+} - x^{b_{j_1}^-}) = \pm (x^{d + b_{j_1}^+} - x^{d + b_{j_1}^-})
\end{equation*}
where we write $m = x^d$ for some $d \in \bb N^r$. Hence 
\begin{equation*}
c^+ - c^- = \pm ((d + b_{j_1}^+) - (d + b_{j_1}^-)) = \pm b_{j_1}
\end{equation*}
as required. 

We prove the general case by induction on $M$. First, up to changing some signs, observe that we can re-order the terms in the expression \ref{eq:sum_to_M} so that $m_1f_{j_1}^+ = x^{c^+}$, hence we can assume that $\sum_{i=1}^{M-1} \epsilon_i m_i f_{j_i}$ is also a pure binomial, say 
\begin{equation*}
\sum_{i=1}^{M-1} \epsilon_i m_i f_{j_i} = x^{c'^+} - x^{c'^-}. 
\end{equation*}
Then by our induction hypothesis we can write $c'^+ - c'^- = \sum_{i=1}^n a'_i b_i$ with $\sum_{i=1}^n \abs{a'_i} \le M-1$. Then
\begin{equation*}
\sum_{i=1}^{M-1} \epsilon_i m_i f_{j_i} = x^{c'^+} - x^{c'^-} = x^{c^+} - x^{c^-} - \epsilon_M m_M (x^{b_{j_M}^+} - x^{b_{j_M}^-}), 
\end{equation*}
and we can (again changing some signs, without loss of generality) assume that $\epsilon_M = +1$ and that $x^{c^-} = m_M x^{b_{j_M}^+} $. Writing $m_M = x^d$, we see
\begin{itemize}
\item
$x^{c^+} = x^{c'^+}$, so $c^+ = c'^+$;
\item $x^{c^-} = x^{d + b_{j_M}^+}$, so $c^- = d + b_{j_M}^+$;
\item $x^{c'^-} = m_Mx^{b_{j_M}^-} = x^{d + b_{j_M}^-}$, so $c'^- = d + b_{j_M}^-$. 
\end{itemize}
Putting these together we see 
\begin{equation*}
c^+ - c^- = c'^+ - c^- = (c'^+ - c'^-) + (b_{j_M}^+ - b_{j_M}^-) = (c'^+ - c'^-) + b_{j_M}, 
\end{equation*}
from which the result is immediate. 
\end{proof}

%




\bibliographystyle{alpha} 
\bibliography{../../prebib.bib}

\end{document}